\documentclass[graybox]{svmult}

\usepackage{amsmath, amssymb}
\usepackage{mathptmx}       
\usepackage{helvet}         
\usepackage{courier}        
\usepackage{makeidx}         
\usepackage{graphicx}        
\usepackage{multicol}        

\usepackage[colorlinks,linkcolor=blue,citecolor=blue,urlcolor=blue]{hyperref}

\newtheorem{prop}[theorem]{Proposition}
\newtheorem{lem}[theorem]{Lemma}
\newtheorem{cor}[theorem]{Corollary}

\title*{Some Notes on Weighted Sum Formulae for Double Zeta Values}
\author{James Wan}
\institute{CARMA, The University of Newcastle, Callaghan, NSW, Australia, \email{james.wan@newcastle.edu.au}}

\begin{document}

\maketitle

\abstract{We present a unified approach which gives completely elementary proofs of three weighted sum formulae for double zeta values. This approach also leads to new evaluations of sums relating to the harmonic numbers, the alternating double zeta values, and the Witten zeta function. We discuss a heuristic for finding or dismissing the existence of similar simple sums. We also produce some new sums from recursions involving the Riemann zeta and the Dirichlet beta functions.}

\section{Introduction}

Multiple zeta values are a natural generalisation of the Riemann zeta function at the positive integers; for our present purposes, we shall only consider multiple zeta values of length 2 (or \textit{double zeta values}), defined for integers $a \ge 2$ and $b \ge 1$ by 
\begin{equation}
\zeta(a,b) = \sum_{n=1}^\infty \sum_{m=1}^{n-1} \frac{1}{n^a m^b}.
\end{equation}
It is rather immediate from series manipulations that 
\begin{equation} \label{easy}
\zeta(a,b)+\zeta(b,a) = \zeta(a)\zeta(b)-\zeta(a+b),
\end{equation}
 thus we can compute in closed form $\zeta(a,a)$, though it is not \textit{a priori} obvious that many other multiple zeta values can be factored into Riemann zeta values. Euler was among the first to study multiple zeta values; indeed, he gave the sum formula (for $s \ge 3$)
\begin{equation} \label{sum1}
 \sum_{j=2}^{s-1} \zeta(j,s-j) = \zeta(s).
\end{equation}
When $s=3$, this formula reduces to the celebrated result $\zeta(2,1)=\zeta(3)$, which has many other proofs \cite{goldbach}. Formula \eqref{sum1} itself may be shown in many ways, one of which uses partial fractions, telescoping sums and change of summation order, which we present in Section \ref{sec2}. Given the ease with which formula \eqref{sum1} may be derived or even experimentally observed (see Section \ref{sec4}), it is perhaps surprising that a similar equation, with `weights' $2^j$ inserted, was only first discovered in 2007 \cite{oz}:
\begin{equation} \label{sum2}
 \sum_{j=2}^{s-1}2^j \zeta(j,s-j) =  (s+1)\zeta(s).
\end{equation}
Formula \eqref{sum2} was originally proven in \cite{oz} using the closed form expression for $\zeta(n, 1)$ (which follows from \eqref{easy} and \eqref{sum1}), together with induction on shuffle relations -- relations arising from iterated integration of generalised polylogarithms which encapsulate the multiple zeta values. Equation \eqref{sum2} has been generalised to more sophisticated weights other than $2^j$ using generating functions, and to lengths greater than 2 (see e.g. \cite{guoxie}). 

In conjunction, \eqref{easy}, \eqref{sum1} and \eqref{sum2} can be used to find a closed form for $\zeta(a,b)$ for $a+b \le 6$. Indeed, it is a result Euler wrote down and first elucidated in \cite{explicit} that all $\zeta(a,b)$ with $a+b$ odd may be expressed in terms of Riemann zeta value (in contrast, $\zeta(5,3)$ is conjectured not reducible to more fundamental constants).

The third weighted sum we will consider is
\begin{equation} \label{sum3}
\sum_{j=2}^{2s-1} (-1)^j \zeta(j,2s-j) = \frac{1}{2}\zeta(2s).
\end{equation}
Given that all known proofs of \eqref{sum2} had their genesis in more advanced areas, one purpose of this note is to show that \eqref{sum2} and the alternating \eqref{sum3} are not intrinsically harder than \eqref{sum1} and can be proven in a few short lines. We use the same techniques in Section \ref{sec3} to give similar identities involving closely related functions. We also observe that some double zeta values sums are related to recursions (or convolutions) satisfied by the Riemann zeta function, a connection which we exploit in Section \ref{sec4}. Lastly, we use such recursions and a reflection formula to produce new results for character sums as defined in \cite{bzb}.

\section{Elementary Proofs} \label{sec2}

In the proofs below, the orders of summation may be interchanged freely, as the sums involved are absolutely convergent.

\begin{proof}[of \eqref{sum2}] We write the left hand side of \eqref{sum2} as 
\[\sum_{j=2}^{s-1} \sum_{m=1}^\infty \sum_{n=1}^\infty \frac{2^j}{n^{s-j}(n+m)^j}.\] 
We consider the 2 cases, $m=n$ and $m \ne n$. In the former case the sum immediately yields $(s-2)\zeta(s)$. In the latter case, we do the geometric sum in $j$ first to obtain
\begin{equation} \label{split}
 \sum_{\stackrel{m,n>0}{m \ne n}} \frac{2^s}{(n^2-m^2)(n+m)^{s-2}}-\frac{4}{(n^2-m^2)n^{s-2}}.
\end{equation}
The first summand in \eqref{split} has antisymmetry in the variables $m, n$ and hence vanishes when summed. 

For the second term in \eqref{split}, we use partial fractions to obtain
\[ \sum_{\stackrel{m>0}{m \ne n}} \frac{1}{m^2-n^2} = \frac{1}{2n} \sum_{\stackrel{m>0}{m \ne n}} \frac{1}{m-n}-\frac{1}{m+n}= \frac{3}{4n^2}, \]
as the last sum telescopes (this is easy to see by first summing up to $m=3n$, then looking at the remaining terms $2n$ at a time).

Therefore, summing over $n$ in the second term of \eqref{split} gives $3 \zeta(s)$. The result follows. \qed
\end{proof}

Our proof suggests that the base `2' in the weighted sum is rather special as it induces antisymmetry. Another special case is obtained by replacing the 2 by a 1, and the same method proves Euler's result.

\begin{proof}[of \eqref{sum1}] We apply the same procedure as in the previous proof and sum the geometric series first, so the left hand side becomes
\begin{align*}
\sum_{m,n>0} \frac{1}{m(m+n)n^{s-2}} - \frac{1}{m(m+n)^{s-1}} & = \sum_{n>0} \frac{1}{n^{s-1}}\sum_{m>0}\biggl(\frac{1}{m}-\frac{1}{m+n}\biggr) - \zeta(s-1,1) \\
& = \sum_{n=1}^\infty \frac{1}{n^{s-1}} \sum_{k=1}^n \frac{1}{k} - \zeta(s-1,1) \\
& = \sum_{n=1}^\infty \frac{1}{n^s} + \sum_{n=1}^\infty \frac{1}{n^{s-1}} \sum_{k=1}^{n-1} \frac{1}{k} - \zeta(s-1,1) \\
& = \zeta(s),
\end{align*}
where we have used partial fractions for the first equality, and telescoping for the second.
\qed
\end{proof}

Likewise we may easily prove the alternating sum \eqref{sum3}:

\begin{proof}[of \eqref{sum3}] We write the left hand side out in full as above, then perform the geometric sum first to obtain
\[ \sum_{m,n>0} \frac{1}{(m+n)(m+2n)n^{2s-2}}-\frac{1}{(m+2n)(m+n)^{2s-1}}.\]
Let $k=m+n$, so we have
\[ \sum_{k>n>0} \frac{1}{k(k+n)n^{2s-2}}-\frac{1}{(k+n)k^{2s-1}}. \]
In the first term, use partial fractions and sum over $k$ from $n+1$ to $\infty$; in the second term, sum over $n$ from 1 to $k-1$. We get
\[ \biggl(\sum_{n>0} \frac{1}{n^{2s-1}} \sum_{k=n+1}^{2n}\frac{1}{k}\biggr) - \biggl(\sum_{k>0} \frac{1}{k^{2s-1}} \sum_{n=k+1}^{2k-1}\frac{1}{n}\biggr). \]
It now remains to observe that if we rename the variables in the second bracket, then the two sums telescope to $\sum_{n>0} 1/(2n^{2s}) = \zeta(2s)/2$. Hence \eqref{sum3} holds. \qed
\end{proof}

\begin{remark} \label{r1} The final sums we shall consider in this section are 
\begin{equation} \label{sum4}
\sum_{j=1}^{s-1} \zeta(2j,2s-2j) = \frac34 \zeta(2s), \qquad \sum_{j=1}^{s-1} \zeta(2j+1,2s-2j-1) = \frac14 \zeta(2s).
\end{equation}
These results were first given in \cite{gkz} and later proven in a more direct manner in \cite{nakamura} using recursion of the Bernoulli numbers. The difference of the two equations in \eqref{sum4} is \eqref{sum3} and the sum is a case of \eqref{sum1}. Therefore, the elementary nature of  \eqref{sum4} is revealed since we have elementary proofs of \eqref{sum1} and \eqref{sum3}.


If we add the first equation in \eqref{sum4} to itself but reverse the order of summation, then upon applying \eqref{easy} we produce the identity
\[ \sum_{j=1}^{s-1} \zeta(2j)\zeta(2n-2j) = \Bigl(n+\frac12\Bigr)\zeta(2s), \]
which is usually derived from the generating function of the Bernoulli numbers $B_n$, since $2(2n)! \, \zeta(2n) = (-1)^{n+1}(2\pi)^{2n} B_{2n}$. \qed
\end{remark}

\section{New Sums} \label{sec3}

We shall see in this section that the elementary methods in Section \ref{sec2} can in fact take us a long way.

\subsection{Witten Zeta Function}

The \textit{Witten zeta function} (also known as the Tornheim, or Mordell, double sum) is defined as
\[ W(r,s,t) = \sum_{n=1}^\infty \sum_{m=1}^\infty \frac{1}{n^r m^s (n+m)^t}.\]
Note that $W(r, s, 0) = \zeta(r)\zeta(s)$ and $W(r,0,t) = W(0,r,t) = \zeta(t,r)$. Due to the simple recursion $W(r,s,t) = W(r-1,s,t+1)+W(r,s-1,t+1)$, when $r,s,t$ are positive integers $W$ may be expressed in terms of Riemann zeta or double zeta values (see e.g. \cite{huard}).

We again emulate the proof of \eqref{sum2} to obtain what seems to be a new sum over $W$. 
\begin{theorem} For integers $a \ge 0$ and $s \ge 3$,
\begin{align} \nonumber
\sum_{j=2}^{s-1} W(s-j,a,j) & = (-1)^a \zeta(s+a)+ (-1)^a \zeta(s+a-1,1)-\zeta(s-1,a+1) \label{thm2} \\ 
& \quad - \sum_{i=2}^{a+1} (-1)^{i+a} \zeta(i)\zeta(s+a-i) \\ \nonumber
& =  \sum_{i=2}^{s-1} \binom{i+a-2}{a} \zeta(i+a,s-i) + \sum_{i=s-a}^{s-1} \binom{i+a-2}{s-3} \zeta(i+a,s-i).
\end{align}
\end{theorem}
\begin{proof}
The second equality follows after some simplification as Witten zeta values can be expressed as double zeta values. For the first equality, we sketch the proof based on that of \eqref{sum2}. Writing the left hand side of \eqref{thm2} as a triple sum, we perform the geometric sum first to produce
\[ \sum_{m,n>0} \frac{(-1)^a}{m^{a+1}n^{s-2}(m+n)}-\frac{(-1)^a}{m^{a+1}(m+n)^{s-1}}. \]
To the first term we apply the partial fraction decomposition
\[ \frac{1}{m^b(m+n)} = \frac{(-1)^b}{n^b(m+n)}+\sum_{i=1}^b \frac{(-1)^{b-i}}{m^i \, n^{b+1-i}}. \]
We recognise the resulting sums as well as the second term above as Riemann zeta and double zeta values. The result follows readily. \qed
\end{proof}

When $a=0$ in \eqref{thm2}, we recover \eqref{sum1}; when $a=1$, we obtain the very pretty formula
\begin{cor}
\begin{equation} \label{witten}
\sum_{j=2}^s W(s-j,1,j) = \zeta(2,s-1),
\end{equation}
\end{cor}
which, by the second equality in \eqref{thm2}, is equivalent to 
\begin{equation} \sum_{j=2}^{s-1} j \, \zeta(j, s-j) = 2\zeta(s) + \zeta(2,s-1) - (s-2) \zeta(s-1,1). \end{equation}
When $a=2$  in \eqref{thm2}, we have
\[ \sum_{j=2}^s W(s-j,2,j) = \zeta(s+2)+\zeta(s+1,1)+\zeta(3,s-1)-\zeta(2,s). \]

A counterpart to \eqref{witten} is the following alternating sum:
\begin{equation} \sum_{j=2}^{2s} (-1)^j \, W(2s+1-j,1,j) = \zeta(2s+1,1) + \frac14 \zeta(2s+2), \end{equation}
and the same procedure can be used to prove this and to provide a closed form for the general case, i.e. the alternating sum of $W(s-j,a,j)$, though we omit the details.

\subsection{Sums Involving the Harmonic Numbers}

The $n$th \textit{harmonic number} is given by $H_n = \sum_{k=1}^n \frac{1}{k}$. If we replace $2s$ by $2s+1$ in the proof of \eqref{sum3} (that is, when the sum of arguments in the double zeta value is odd instead of even), then we obtain
\begin{equation}
 \frac{5}{2}\zeta(2s+1)+2\zeta(2s,1)+\sum_{j=2}^{2s} (-1)^j \zeta(j,2s+1-j) = 2 \sum_{n=1}^\infty \frac{H_{2n}}{n^{2s}}.
\end{equation}
Combined with known double zeta values, we  can evaluate the right hand side, giving
\[ \sum_{n=1}^\infty \frac{H_{2n}}{n^4} = \frac{37}{4}\zeta(5)-\frac{2}{3}\pi^2\zeta(3), \]
etc, in agreement with results obtained via Mellin transform and generating functions in \cite{bzb} (in whose notation such sums are related to $[2a,1](2s,1)$ -- this notation is explained in Section \ref{sec4}). Indeed, replacing our right hand side with results in \cite{bzb}, we have:
\begin{equation}
 \sum_{j=2}^{2s} (-1)^j \zeta(j,2s+1-j)  = (4^s-s-2)\zeta(2s+1) - 2 \sum_{k=1}^{s-1}(4^{s-k}-1)\zeta(2k)\zeta(2s+1-2k). 
\end{equation}
Similarly, using weight $\frac12$ (instead of 2), we have another new result:
\begin{lem} For integer $s \ge 3$,
\begin{align} \nonumber \sum_{j=2}^{s-1} 2^{1-j}\zeta(j,s-j) & = (2^{1-s}-1)\bigl(\zeta(s-1,1)- 2\log(2)\zeta(s-1)\bigr) \\
& \quad +  (2^{2-s}-1)\zeta(s) + \sum_{n=0}^\infty \frac{H_n}{(2n+1)^{s-1}}. \label{hn}
\end{align}
\end{lem}
Therefore, we may produce evaluations such as
\begin{align*}
\sum_{n=0}^\infty \frac{H_n}{(2n+1)^4} & = \frac{372\zeta(5)-21\pi^2\zeta(3)-2\pi^4\log(2)}{96}, \\
\sum_{n=0}^\infty \frac{H_n}{(2n+1)^5}  & = \frac{\pi^6-294\zeta(3)^2-744\log(2)\zeta(5)}{384}.
\end{align*}
Indeed, in \eqref{hn} the harmonic number sum relates to the functions $[2a,1]$ and $[2a,2a]$ in \cite{bzb}, and when $s$  is odd, we use their closed forms to simplify \eqref{hn}:
\begin{align} \nonumber
  \sum_{j=2}^{2s} 2^{1-j} \zeta(j,2s+1-j) & = (s-1+2^{1-2s})\zeta(2s+1) \\
& \quad - \sum_{k=1}^{s-1} (4^{-k}-4^{-s})(4^k-2) \zeta(2k) \zeta(2s+1-2k).  
\end{align}
On the other hand, if we chose even $s$ in \eqref{hn}, then $[2a,1]$, $[2a,2a]$ seem not to simplify in terms of more basic constants, though below we manage to find a closed form for their difference  (the proof is more technical than other results in this section). Combined with \eqref{hn}, we have
\begin{theorem} \label{thmhn} For integer $s \ge 2$,
\begin{align}  
\sum_{n=0}^\infty \frac{H_n}{(2n+1)^{2s-1}} & = (1-4^{-s}) (2s-1)\zeta(2s) - (2-4^{1-s}) \log(2) \zeta(2s-1) \nonumber \\
 & \quad + (1-2^{-s})^2 \zeta(s)^2  - \sum_{k=2}^s 2(1-2^{-k})(1-2^{k-2s}) \zeta(k)\zeta(2s-k); \\
\sum_{j=2}^{2s-1} 2^{1-j}\zeta(j,2s-j) & = \frac{1}{2}(1-2^{1-s})^2 \zeta(s)^2 + \frac{1}{2}(2^{3-2s}+2s-3) \zeta(2s)  \nonumber \\
 & \quad  - \sum_{k=2}^s (2^{k-1}-1)(2^{1-k}-4^{1-s})\zeta(k)\zeta(2s-k).
\end{align}
\end{theorem}
\begin{proof}
We only need to prove the first equality as the second follows from \eqref{hn}; to achieve this we borrow techniques from \cite{bzb}.

Using the fact that the harmonic number sum is $2([2a,1](2s-1,t)-[2a,2a](2s-1,t))$ in the notation of \cite{bzb}, we use the results therein (obtained using Mellin transforms)  to write down its integral equivalent:
\[ \sum_{n=0}^\infty \frac{H_n}{(2n+1)^{2s-1}} = \int_0^1 \frac{\log(x)^{2s-2}\log(1-x^2)}{\Gamma(2s-1)(x^2-1)} \,\mathrm{d}x. \]
We denote its generating function by $F(w)$, and after interchanging orders of summation and integration, we obtain
\begin{align*}
 F(w) & := \sum_{s=2}^\infty \biggl[ \int_0^1 \frac{\log(x)^{2s-2}\log(1-x^2)}{\Gamma(2s-1)(x^2-1)} \,\mathrm{d}x \biggr] w^{2s-2} \\
 & = \int_0^1 \frac{x^{-w}(x^w-1)^2 \log(1-x^2)}{2(x^2-1)} \,\mathrm{d}x \\
 & = -\frac{1}{2} \int_0^1 \frac{\mathrm{d}}{\mathrm{d}q}\Bigl[x^{-w}(x^w-1)^2(1-x^2)^{q-1}\Bigr]_{q=0} \,\mathrm{d}x.
\end{align*}
Next, we interchange the order of differentiation and integration; the result is a Beta integral which evaluates to:
\begin{align*}
 F(w) & = \frac{1}{4} \frac{\mathrm{d}}{\mathrm{d}q}\biggl[ \frac{2 \Gamma(1/2) \Gamma(q)}{\Gamma(q+1/2)}-\frac{\Gamma((1-w)/2)\Gamma(q)}{\Gamma((1-w)/2+q)}-\frac{\Gamma((1+w)/2)\Gamma(q)}{\Gamma((1+w)/2+q)} \biggr]_{q=0} \\
 & = \frac{1}{8}\biggl[ 8\log(2)^2-\pi^2+\pi^2 \sec^2\Bigl(\frac{\pi w}{2}\Bigr) - \Bigl[\psi\Bigl(\frac{1-w}{2}\Bigr)+\gamma\Bigr]^2 - \Bigl[\psi\Bigl(\frac{1+w}{2}\Bigr)+\gamma\Bigr]^2 \biggr],
\end{align*} 
where $\psi$ denotes the digamma function and $\gamma$ is the Euler-Mascheroni constant.  The desired equality follows using the series expansions
\begin{align*}
-\psi\Bigl(\frac{1-w}{2}\Bigr) & = \gamma+2\log(2) + \sum_{k=1}^\infty (2-2^{-k}) \zeta(k+1)w^k, \\
\frac{\pi^2}{2} \sec^2\Bigl(\frac{\pi w}{2}\Bigr) & = \sum_{k=0}^\infty (4-4^{-k})(2k+1)\zeta(2k+2)w^{2k}. \qquad \qed
\end{align*}
\end{proof}

\begin{remark}
Thus Theorem \ref{thmhn}, together with \cite{bzb}, completes the evaluation of \[\sum_{n=0}^\infty \frac{H_n}{(2n+1)^s}\] in terms of well known constants for integer $s \ge 2$. In \cite[theorem 6.5]{apostol} it is claimed that said sum may be evaluated in terms of Riemann zeta values alone, but the claim is unsubstantiated by numerical checks and notably the constant $\log(2)$ is missing from the purported evaluation. \qed
\end{remark}

\subsection{Alternating Double Zeta Values}

The \textit{alternating} double zeta values $\zeta(a, \overline{b})$ are defined as
\[ \zeta(a,\overline{b}) = \sum_{n=1}^\infty \sum_{m=1}^{n-1} \frac{1}{n^a}\frac{(-1)^{m-1}}{m^b}, \]
with $\zeta(\overline{a},b)$ and $\zeta(\overline{a}, \overline{b})$ defined similarly (the bar indicates the position of the $-1$). In \cite{bzb}, explicit evaluations of $\zeta(s,\overline{1}),\zeta(\overline{2s},1)$ and $\zeta(\overline{2s},\overline{1})$ are given in terms of Riemann zeta values and $\log(2)$; in \cite{explicit}, it is shown that $\zeta(a,\overline{b})$ etc with $a+b$ odd may be likewise reduced; small examples include (see also \cite{goldbach} for the first one)
 \[ \zeta(\overline{2},1) = -\frac{\zeta(3)}{8}, \quad \zeta(2, \overline{1}) = \frac{\pi^2 \log(2)}{4} - \zeta(3), \quad \zeta(\overline{2},\overline{1}) = \frac{\pi^2 \log(2)}{4}-\frac{13\zeta(3)}{8}. \]

Again, if we follow closely the proof of \eqref{sum1}, we arrive at new summation formulae such as
\begin{equation}
\sum_{j=2}^{s-1} \zeta(j,\overline{s-j}) = (1-2^{1-s})\zeta(s)+\zeta(\overline{s-1},1)+\zeta(\overline{s-1},\overline{1}),
\end{equation}
and so on. When $s$ is odd, we simplify the right hand side using results in \cite{bzb}, thus:
\begin{align*}
\sum_{j=2}^{2s} \zeta(j,\overline{2s+1-j}) & = 2(1-4^{-s})\log(2)\zeta(2s)-\zeta(2s+1) \\
& \quad +\sum_{k=1}^{s-1} (2^{1-2k}-4^{k-s})\zeta(2k)\zeta(2s+1-2k), \\
\sum_{j=2}^{2s} (-1)^j \zeta(j,\overline{2s+1-j}) & = 2(1-4^{-s})\log(2)\zeta(2s) + \Bigl((s+1)4^{-s}-\frac12-s\Bigr)\zeta(2s+1) \\
& \quad + \sum_{k=1}^{s-1} (1-4^{k-s})\zeta(2k)\zeta(2s+1-2k).
\end{align*}
The last two formulae may be added or subtracted to give sums for even or odd $j$'s, for instance
\begin{equation} \sum_{j=1}^{s-1} \zeta(2j+1,\overline{2s-2j}) = \Bigl(\frac{2s-1}{4}-\frac{s+1}{2^{2s+1}}\Bigr)\zeta(2s+1)-\sum_{k=1}^{s-1} \Bigl(\frac12-\frac{1}{4^k}\Bigr)\zeta(2k)\zeta(2s+1-2k). 
\end{equation}

With perseverance, we may produce a host of similar identities for the three alternating double zeta functions. We only give some examples below; as they have similar proofs, we omit the details. 

When we evaluate $\sum_{j=2}^{s-1} (\pm 1)^j \zeta(\overline{j},s-j)$, we get for example
\begin{align*} \sum_{j=2}^{s-1}\zeta(\bar{j},s-j) &= (1-2^{1-s})\bigl(\zeta(s)+\zeta(s-1,1)-2\log(2)\zeta(s-1)\bigr) \\
& \quad -\zeta(\overline{s-1},1) -\sum_{n=0}^\infty \frac{H_n}{(2n+1)^{s-1}},
\end{align*}
and by applying \eqref{hn} to the results, we obtain
\begin{align}
 \sum_{j=2}^{2s} \bigl(2^{1-j}\zeta(j,2s+1-j) + \zeta(\overline{j},2s+1-j)\bigr) & = 4^{-s}\zeta(2s+1) - \zeta(\overline{2s},1), \\
2 \sum_{j=1}^s \zeta(\overline{2j},2s+1-2j) & = \frac{4^{-s}-1}{2}\zeta(2s+1)-\zeta(\overline{2s},1),
\end{align}
where the right hand side of both equations may be reduced to Riemann zeta values by results in \cite{bzb}.

Likewise, for $\zeta(\overline{j},\overline{s-j})$ we may deduce
\begin{align} \nonumber
\frac{1}{2} \sum_{j=2}^{2s} \zeta(\overline{j},\overline{2s+1-j}) & = (1-4^{-s})\log(2)\zeta(2s)-\frac{2s(2^{2s+1}-1)-1}{4^{s+1}}\zeta(2s+1) \\
& \quad +\sum_{k=1}^{s-1}(4^k-1)(4^{-k}-4^{-s})\zeta(2k)\zeta(2s+1-2k); \\ \nonumber
\sum_{j=1}^{s-1} \zeta(\overline{2j+1},\overline{2s-2j}) & = \Bigl(\frac12 (4^s-3s-2)+4^{-s}(s+1)\Bigr)\zeta(2s+1) \\
& \quad - \sum_{k=1}^{s-1} 4^{-(s+k)}(4^s-4^k)^2 \zeta(2k)\zeta(2s+1-2k). 
\end{align}

Therefore, for the sums of the three alternating double zeta values, we have succeeded in giving closed forms when $s$ (the sum of the arguments) is odd and the summation index $j$ is odd, even, or unrestricted; it is interesting to compare this to the non-alternating case, whose sum is simpler when $s$ is even (see Remark \ref{r1}). A notable exception is the following formula, whose proof is similar to that of \eqref{sum3}:
\begin{theorem}  \label{thmmm} For integer $s \ge 2$, 
\begin{equation}
4\sum_{j=1}^{s-1} \zeta(\overline{2j},\overline{2s-2j}) = (4^{1-s}-1)\zeta(2s). 
\end{equation}
\end{theorem}

\begin{remark}
Though it is believed that $\zeta(\overline{s},1)$ and $\zeta(\overline{s},\overline{1})$  cannot be simplified in terms of well known constants for odd $s$, their difference can (this situation is analogous to Theorem \ref{thmhn} and can be proven using the same method):
\[ \zeta(\overline{s},1)  - \zeta(\overline{s},\overline{1}) = (1-2^{-s})\bigl(s\zeta(s+1)-2\log(2)\zeta(s)\bigr) - \sum_{k=1}^{s-2}(1-2^{-k})\zeta(k+1)\zeta(s-k). \]

Moreover, some of the sums involving $\zeta(\overline{s},1)$ and $\zeta(\overline{s},\overline{1})$ are much neater when the summation index $j$ starts from 1 instead of 2, for instance
\begin{align*}
\sum_{j=1}^{s-1} \zeta(\overline{j},s-j) & = (2^{2-s}-1)\log(2)\zeta(s-1)-\zeta(\overline{s-1},1), \\
\sum_{j=1}^{s-1} \zeta(\overline{j},\overline{s-j}) & = \zeta(s-1,\overline{1})-\log(2)\zeta(s-1), \\
2\sum_{j=1}^{2s} (-1)^j \zeta(\overline{j},2s+1-j) & = (2-4^{1-s})\log(2)\zeta(2s)-(1 - 4^{-s})\zeta(2s+1). \quad \qed
\end{align*}
\end{remark}

We wrap up this section with a surprising result, an alternating analog of \eqref{sum2}:
\begin{theorem} For integer $s \ge 3$,
\begin{equation} \label{zmp}
\sum_{j=2}^{s-1} 2^j \zeta(\overline{j},s-j) = (3-2^{2-s}-s) \zeta(s).
\end{equation}
\end{theorem}
\begin{proof}
The proof is very similar to that of \eqref{sum2}: we write the left hand side as a triple sum and first take care of the $m=n$ case. Then we sum the geometric series to obtain
\[ \sum_{\stackrel{m,n>0}{m \ne n}} \frac{(-1)^{m+n}2^s}{(m-n)(m+n)^{s-1}}-\frac{4(-1)^{m+n}}{(m-n)(m+n)n^{s-2}}. \]
The first term vanishes due to antisymmetry, and the second term telescopes due to
\[ \sum_{\stackrel{m>0}{m \ne n}} \frac{(-1)^m}{m^2-n^2} = \frac{2+(-1)^n}{4n^2}. \]
Now summing over $n$ proves the result. \qed
\end{proof}
With \eqref{zmp} and results in \cite{bzb}, we can evaluate $\zeta(\overline{a},b)$ etc with $a+b=4$, for instance
\[ \zeta(\overline{2},2) = \frac{\log(2)^4}{6}-\frac{\log(2)^2 \pi^2}{6} +\frac{7\log(2)\zeta(3)}{2}-\frac{13\pi^4}{288} + 4 \mathrm{Li}_4\Bigl(\frac12\Bigr), \]
where $\mathrm{Li}_4$ is the polylogarithm of order 4.

\section{More Sums from Recursions} \label{sec4}

In this section we first provide some experimental evidence which suggests that the sums in Section \ref{sec2} (almost) exhaust all `simple' and `nice' sums in some sense. We then use a simple procedure which may be used to produce more weighted sums of greater complexity but of less elegance.

\subsection{Experimental Methods}

It is a curiosity why \eqref{sum2} had not been observed empirically earlier. As we can express all $\zeta(a,b)$ with $a+b \le 7$ in terms of the Riemann zeta function, it is a simple matter of experimentation to try all combinations of the form
\begin{equation} \label{exp} \sum_j (a \cdot b^j + c^s \cdot d^j) \zeta(j,s-j) = f(s) \zeta(s), \end{equation}
with $j$ or $s$ being even, odd or any integer (so there are 9 possibilities), $a, b, c, d \in \mathbb{Q}$, and $f: \mathbb{N} \to \mathbb{Q}$ is a (reasonable) function to be found. 

Now if we assume that $\pi, \zeta(3), \zeta(5), \zeta(7), \ldots$ are algebraically independent over $\mathbb{Q}$ (which is widely believed to be true, though proof-wise we are a long way off, for instance, apart from $\pi$ only $\zeta(3)$ is known to be irrational -- see \cite{alf}), then we can substitute a few small values of $s$ into \eqref{exp} and solve for $a, b, c, d $ in that order.

For instance, assuming a formula of the form $\sum_{j=2}^{s-1} a^j \zeta(j,s-j) = f(s) \zeta(s)$ holds, using $s=5$ forces us to conclude that $a=1$ or $a=2$.

Indeed, when we carry out the experiment outlined above, it is revealed that the sums \eqref{sum1}, \eqref{sum2}, \eqref{sum3} are essentially the only ones in the form of \eqref{exp}, except for the case
\[ \sum_{j=1}^{s-1} (d^j + d^{s-j})\zeta(2j,2s-2j), \]
(note the factor in front of $\zeta$ has to be invariant under $j \mapsto s-j$). Here, the choice of $d=4$ leads 
\[ \sum_{j=1}^{s-1} (4^j+4^{s-j})\zeta(2j,2s-2j) =  \Bigl(s+\frac43+\frac23 4^{s-1}\Bigr)\zeta(2s),\]
a result which first appeared in \cite{nakamura} and was proven using the generating function of Bernoulli polynomials. \\

Sums of the form 
\[ \sum_j p(s,j) \zeta(j,s-j) = f(s) \zeta(s), \]
where $p$ is a non-constant 2-variable polynomial with rational coefficients, can also be subject to experimentation. If the degree of $p$ is restricted to 2, then $j(s-j) \zeta(2j, 2s-2j)$ is the only candidate which can give a closed form. Indeed, this sum was essentially considered in \cite{nakamura}, using an identity found in \cite{shimura}: 
\[ 6 \sum_{j=2}^{s-2}(2j-1)(2s-2j-1)\zeta(2j)\zeta(2s-2j) = (s-3)(4s^2-1)\zeta(2s).\] 
The identity was first due to Ramanujan \cite[chapter 15, formula (14.2)]{rama} (and hence not original as claimed in \cite{shimura}). Applying \eqref{easy}, the result can be neatly written as
\begin{equation} \label{naka}
 \sum_{j=2}^{s-2} (2j-1)(2s-2j-1) \zeta(2j,2s-2j) = \frac{3}{4}(s-3)\zeta(2s).
\end{equation}

Searches for `simple' weighted sums of length 3 multiple zeta values, and for $q$-analogs of \eqref{sum2}, have so far proved unsuccessful (except for \cite{shencai} which contains a generalisation of \eqref{sum3}). \\

\subsection{Recursions of the Zeta Function} 

We observe that any recursion of the Riemann zeta values -- or of Bernoulli numbers -- of the form
\[ \sum_j  g(s,j) \zeta(2j)\zeta(2s-2j) \]
for some function $g$ would lead to a sum formula for double zeta values, due to \eqref{easy}. This was the idea behind \eqref{naka} and was also hinted at in Remark \ref{r1}. We flesh out the details in some examples below.

One such recursion is formula (14.14) in chapter 15 of \cite{rama}, which can be written as
\begin{align} \nonumber
& \quad \sum_{j=1}^{n-1} j(2j+1)(n-j)(2n-2j+1) \zeta(2j+2)\zeta(2n-2j+2)  \\
& = \frac{1}{60}(n+1)(2n+3)(2n+5)(2n^2-5n+12) \zeta(2n+4) - \frac{\pi^4}{15}(2n-1)\zeta(2n).
\end{align}
Upon applying \eqref{easy} to the recursion, we obtain the new sum:
\begin{theorem} For integer $n \ge 4$,
\begin{align} \nonumber 
& \quad \sum_{j=2}^{n-2} (j-1)(2j-1)(n-j-1)(2n-2j-1) \zeta(2j,2n-2j) \\
& = \frac{3}{8}(n-1)(3n-2)\zeta(2n)-3(2n-5)\zeta(4)\zeta(2n-4). 
\end{align}
\end{theorem}

Next, we use a result from \cite{miki}, which states
\begin{equation} \label{b1}
 \sum_{k=1}^{n-1} \biggl[1-\binom{2n}{2k}\biggr] \frac{ B_{2k}B_{2n-2k}}{(2k)(2n-2k)}  = \frac{H_{2n}}{n} B_{2n}.
\end{equation}
We apply \eqref{easy} to the left hand side to obtain a sum of double zeta values; unfortunately one term of the sum involves $\sum_{k=1}^{n-1} (2k-1)!(2n-2k-1)!$ which has no nice closed form. On the other hand, a twin result in \cite{yuri} gives
\begin{equation} \label{b2}
 \sum_{k=1}^{n-2} \biggl[ n-\binom{2n}{2k} \biggr] B_{2k}B_{2n-2k-2} = (n-1)(2n-1)B_{2n-2}. 
\end{equation}
When we apply \eqref{easy} to it, we end up with a sum involving $\sum_{k=1}^{n-2} (2k)!(2n-2k-2)!$, which again has no nice closed form.

Yet, it is straight-forward to show by induction that
\[ \sum_{k=0}^m \frac{(-1)^k}{\binom{n}{k}} = \frac{(n+1)!+(-1)^m (m+1)!(n-m)!}{(n+2)n!}, \]
hence, when $m=n$, the sum vanishes if $n$ is odd and equates to $2(n+1)/(n+2)$ when $n$ is even. In other words, $\sum_{k=0}^{2n} (-1)^k k!(2n-k)!$ has a closed form, and accordingly we subtract the sums obtained from \eqref{b1} and \eqref{b2} to produce:
\begin{prop} For integer $n \ge 2$,
\begin{align} \nonumber
& \quad \sum_{k=1}^{n-1} \biggl\{ \biggl[1-\frac{1}{\binom{2n+2}{2k}}\biggr]\frac{n+1}{k(n+1-k)} \zeta(2k,2n+2-2k) +  \frac{\zeta(2n+2)}{\zeta(2n)} \times \\ 
\nonumber & \quad \biggl[\frac{2}{(2n+1)\binom{2n}{2k}} - \frac{(2k-n)^2+(n+1)(n+2)}{(n-k+1)(2n-2k+1)(k+1)(2k+1)} \biggr]\zeta(2k,2n-2k) \biggr\}\\
& = 3 \biggl[ H_{2n-1}-H_{n-1}-\frac{2n^2+n+1}{2n(2n+1)} \biggr] \zeta(2n+2) - \frac{2n+3}{2n+1}\zeta(2n,2).
\end{align}
\end{prop}

Our next result uses equation (7.2) recored in \cite{dilcher}, whose special case gives:
\begin{equation} \label{ad} B_n^2 + \frac{B_{2n}}{\binom{2n}{n}} = \frac{4n\, n!}{(2\pi)^{2n}} \sum_{k=0}^{[n/2]} \frac{(2n-2k)!}{(n-k)(n-2k)!} \, \zeta(2k)\zeta(2n-2k).
\end{equation}
Upon applying \eqref{easy} and much algebra, we arrive at:
\begin{prop} For integer $n \ge 2$,
\begin{align*} \nonumber
 & \quad \sum_{k=1}^{n-1} \frac{(n+|n-2k|)!}{(n+|n-2k|)|n-2k|!} \, \zeta(2k,2n-2k) \\
 & = \biggl[\frac{(1 + (-2)^{1-n})(n-1)!}{4} + \frac{3(2n-1)!}{2n!} - \frac{(2n+1)!}{n(n+1)!} \,  _2F_1\biggl({{1,2n+2} \atop {n+2}} \biggl| -1\biggr) \biggr] \zeta(2n),
\end{align*}
where $_2F_1$ is the Gaussian hypergeometric function.
\end{prop}
\begin{proof}
The only non-trivial step to check here is that the claimed $_2F_1$ is produced when we sum the fraction in \eqref{ad}; that is, we wish to prove the claim
\[  \sum_{k=0}^{[n/2]} \frac{(-1)^n n}{2(n-k)}\binom{2n-2k}{n-2k}= \frac{1}{(1-x)^{n+1}} - x^m \binom{n+m}{n} \, _2F_1\biggl({{1,n+m+1} \atop {m+1}} \biggl| x\biggr)\bigg|_{x=-1,m=n+1}. \]
We observe that for $x$ near the origin the right hand side is simply $\sum_{k=0}^{m-1} x^k \binom{n+k}{k}$, as they have the same recursion and initial values in $m$, hence when $x=-1$ they also agree by analytic continuation. This sum (in the limit $x=-1, \, m=n+1$), as a function of $n$, also satisfies the recursion
\[ 4f(n) - 2 f(n-1) = 3(-1)^n \binom{2n}{n}, \quad f(1) = -1, \]
which is the same recursion for the sum on the left hand side of the claim -- as may be checked using Celine's method \cite{ab}. Thus equality is established.
\qed
\end{proof}

\begin{remark}
It is clear that a large number of (uninteresting) identities similar to the those recorded in the two propositions may be easily produced. Using \cite[(7.2) with $k=n+1$]{dilcher}, for instance, a very similar proof to the above gives
\begin{align*} 
& \quad \sum_{k=1}^{n-1} \frac{4(n+1+|n-2k|)!}{n!(1+|n-2k|)!}\zeta(2k,2n-2k) = (1+(-1)^n)(1-n)\zeta(n)^2 + \zeta(2n) \times\\
& \biggl\{n+3+(-1)^n(n-1+2^{-n})+2 \binom{2n+1}{n}\biggl[3-\frac{4(2n+3)}{n+1}\, _2F_1\biggl({{1,2n+4} \atop {n+3}} \biggl| -1\biggr)\biggr]\biggr\}.
\end{align*}
Care must be exercised when consulting the literature, however, as the author found in the course of this work that many recorded recursions of the Bernoulli numbers (or of the even Riemann zeta values) are in fact combinations and reformulations of the formula behind \eqref{naka} and the basic identity appearing in Remark \ref{r1}.
\qed
\end{remark}

\subsection{The Reflection Formula}

Formula \eqref{easy} is but a special case of a more general \textit{reflection formula}. To state the reflection formula, we will need some notation from \cite{bzb}, which we have tried to avoid until now to keep the exposition elementary.

Let $\chi_p(n)$ denote a 4-periodic function on $n$; for different $p$'s we tabulate values of $\chi_p$ below:

\begin{center}
\begin{tabular}{|c||c|c|c|c|} \hline
$p \backslash n$ & $ \ 1 \ $ & $ \ 2 \ $ & $ \ 3 \ $ & $ \ 4 \ $ \\ \hline \hline
1 & 1 & 1 & 1 & 1 \\ \hline
$2a$ & 1 & 0 & 1 & 0 \\ \hline
$2b$ & 1 & $-1$ & 1 & $-1$ \\ \hline
$-4$ & 1 & 0 & $-1$ & 0 \\ \hline
\end{tabular}
\end{center}

We now define the series $L_p$ by
\[ L_p(s) = \sum_{n=1}^\infty \frac{\chi_p(n)}{n^s}, \]
and $L_{pq}(s)$ means $\sum_{n>0} \chi_p(n)\chi_q(n)/n^s$. Finally, we define \textit{character sums}, which generalise the double zeta values, by
\begin{equation} \label{gen}
[p,q](s,t) = \sum_{n=1}^\infty \sum_{m=1}^{n-1}\frac{\chi_p(n)}{n^s}\frac{\chi_q(m)}{m^t}.
\end{equation} 
In this notation, $\zeta(s,t) = [1,1](s,t), \, \zeta(s,\overline{t}) = [1,2b](s,t)$, etc. We can now state the reflection formula \cite[equation (1.7)]{bzb}:
\begin{equation} \label{reflection}
[p,q](s,t) + [q,p](t,s) = L_p(s)L_q(t) - L_{pq}(s+t).
\end{equation}

\begin{remark}  With the exception of $\chi_{2b}$, $\chi_p$ are examples of Dirichlet characters and $L_p$ are the corresponding Dirichlet series. Indeed, 
\begin{align*}  L_1(s)  =  \zeta(s), \quad & L_{2a}(s)   =  (1-2^{-s})\zeta(s) = \lambda(s), \\
 L_{2b}(s)   =   (1-2^{1-s})\zeta(s) = \eta(s), \quad & L_{-4}(s)  =  \beta(s),
\end{align*}
where the last three are the Dirichlet lambda, eta and beta functions respectively. 

Moreover, $2(2n)! \, \beta(2n+1) = (-1)^n (\pi/2)^{2n+1} E_{2n}$ for non-negative integer $n$, where $E_n$ denotes the $n$th Euler number. Using generating functions, one may deduce convolution formulae for the Euler numbers, an example of which is
\[ \sum_{k=0}^{n-2} \binom{n-2}{k}E_k E_{n-2-k} = 2^n(2^n-1)\frac{B_n}{n}. \]
Many of our results in the previous sections would look neater had we used $\lambda(s)$ and $\eta(s)$ instead of $\zeta(s)$. \qed
\end{remark}

Using the standard convolution formulae of the Bernoulli and the Euler polynomials, and aided by the reflection formula \eqref{reflection}, we can produce the following sums for $[p,q](s,t)$ as we did for $\zeta(s,t)$.
\begin{theorem} Using the notation of \eqref{gen}, we have, for integer $n \ge 2$,
\begin{align} \nonumber
& \quad 2\sum_{k=1}^{n-1} [1,2b](2k,2n-2k)+[2b,1](2k,2n-2k) = -4\sum_{k=1}^{n-1}[2b,2b](2k,2n-2k)  \\
& = (1-4^{1-n})\zeta(2n); \label{altn}
\end{align}
\begin{align} \nonumber
& \quad 4\sum_{k=1}^{n-1} [2a,2a](2k,2n-2k) =-4\sum_{k=0}^{n-1} [-4,-4](2k+1,2n-1-2k) \\
 & = \sum_{k=1}^{n-1} [1,2a](2k,2n-2k)+[2a,1](2k,2n-2k)  = (1-4^{-n})\zeta(2n);
\end{align}
\begin{align} \nonumber 
& \quad \sum_{k=1}^n [2a,-4](2k,2n+1-2k)+[-4,2a](2n+1-2k,2k) \\
& = \sum_{k=1}^{n-1} [2a,2b](2k,2n-2k)+[2b,2a](2n-2k,2k) = 0; \label{zer}
\end{align}
\begin{align} \nonumber
& \qquad 2\sum_{k=1}^n [1,-4](2k,2n+1-2k)+[-4,1](2n+1-2k,2k) \\
& = -2\sum_{k=1}^n [2b,-4](2k,2n+1-2k)+[-4,2b](2n+1-2k,2k) \nonumber \\
& = \beta(2n+1)+(16^{-n}-2^{-1-2n})\pi \,\zeta(2n), 
\end{align}
\end{theorem}

We note that \eqref{altn} concerns the alternating double zeta values studied in Section \ref{sec3} (c.f. the more elementary Theorem \ref{thmmm}). As mentioned before, the identities above rest on well-known recursions, for instance the second equality in \eqref{zer} is equivalent to the recursion
\[ \sum_{k=1}^n \beta(2n+1-2k)\lambda(2k) = n\beta(2n+1).\]
Also, the many pairs of equalities within each numbered equation in the theorem are not all coincidental but stem from the identity in \cite{bzb}:
\[ [1,q]+[2b,q]=2[2a,q], \]
where $q = 1, 2a, 2b$ or $-4$.

Moreover, one can show the following equations; as character sums are not the main object of our study, we omit the details:
\begin{align*}
& \quad \sum_{k=1}^n 4^{-k}\bigl([-4,1](2n+1-2k,2k)+[1,-4](2k,2n+1-2k)\bigr)  \\
& = \frac{1+2^{1-2n}}{6}\beta(2n+1) + (2^{-1-4n}-4^{-1-n}) \pi \, \zeta(2n), \\
& \quad \sum_{k=1}^{n-1} 4^k \bigl([2a,1](2k,2n-2k)+[1,2a](2n-2k,2k)\bigr) \\
& = \frac{(1-4^{-n})(8+4^n)}{6}\zeta(2n).
\end{align*}
Since there is an abundance of recursions involving the Bernoulli and the Euler numbers, many more such identities may be produced using the reflection formula.

\bigskip

\begin{acknowledgement}
The author wishes to thank John Zucker and Wadim Zudilin for illuminating discussions.
\end{acknowledgement}

\end{document}